\documentclass[preprint,12pt]{elsarticle}% Use utf-8 encoding for foreign characters
\usepackage[utf8]{inputenc}

\usepackage{fullpage}

\usepackage{amsmath}
\usepackage{amssymb}
\usepackage{latexsym}
\usepackage{amsthm}
\usepackage[retainorgcmds]{IEEEtrantools}

\usepackage{tikz}
\usetikzlibrary{arrows,calc}

\DeclareMathOperator{\obs}{obs}
\usepackage[capitalize]{cleveref}

\newtheorem{theorem}{Theorem}
\newtheorem*{theorem*}{Theorem}
\newtheorem{lemma}[theorem]{Lemma}
\newtheorem{observation}[theorem]{Observation}
\newtheorem{proposition}[theorem]{Proposition}
\newtheorem{corollary}[theorem]{Corollary}

\newtheorem*{question}{Question}

\usepackage{minitoc}

\usepackage{ifpdf}

\begin{document}

\begin{frontmatter}

%% Texto Principal
%\title{Double circuits in bicircular matroids\tnoteref{t1}}
\title{Full-homomorphisms to paths and cycles}

%\tnotetext[t1]{This research was carried out during a visit of the first author
%at FernUniversit\"at in Hagen and supported by DAAD [grant number 57552339].}

\author[FC]{Santiago Guzm\'an-Pro}\corref{cor1}
\ead{sanguzpro@ciencias.unam.mx}

\address[FC]{Facultad de Ciencias\\
Universidad Nacional Aut\'onoma de M\'exico\\
Av. Universidad 3000, Circuito Exterior S/N\\
C.P. 04510, Ciudad Universitaria, CDMX, M\'exico}

\cortext[cor1]{Corresponding author}

\begin{abstract}
A full-homomorphism between a pair of graphs  is a vertex
mapping that preserves adjacencies and non-adjacencies. For a fixed
graph $H$, a full $H$-colouring is a full-homomorphism of $G$ to $H$.
A minimal $H$-obstruction is a graph that does not admit a full $H$-colouring,
such that every proper induced subgraph of $G$ admits a full $H$-colouring.
Feder and Hell proved that for every graph $H$ there is a finite number
of minimal  $H$-obstructions. We begin this work by describing all minimal
obstructions of paths. Then, we study minimal obstructions of regular graphs
to propose a description of minimal obstructions of cycles.
As a consequence of these results, we observe that for each path $P$ and
each cycle $C$, the number of minimal $P$-obstructions and $C$-obstructions
is  $\mathcal{O}(|V(P)|^2)$ and $\mathcal{O}(|V(C)|^2)$, respectively. 
Finally, we propose some problems regarding the largest minimal
$H$-obstructions, and the number of minimal $H$-obstructions.

\end{abstract}


\end{frontmatter}

%%%%%%%%%%%%%%%%%%%%%%%%%%%%%%%%%%%%%%%
%%%%%%%%%%%%%%%%%%%%%%%%%%%%%%%%%%%%%%%
%%%%%%%%%%%%%%%%%%%%%%%%%%%%%%%%%%%%%%%
%%%%%%%%%%%%%%%%%%%%%%%%%%%%%%%%%%%%%%%
%%%%%%%%% Introduction %%%%%%%%%%%%%%%%%%%%%%%%
%%%%%%%%%%%%%%%%%%%%%%%%%%%%%%%%%%%%%%%
%%%%%%%%%%%%%%%%%%%%%%%%%%%%%%%%%%%%%%%
%%%%%%%%%%%%%%%%%%%%%%%%%%%%%%%%%%%%%%%
%%%%%%%%%%%%%%%%%%%%%%%%%%%%%%%%%%%%%%%

\section{Introduction}

All graphs considered in this work are finite graphs with no parallel edges. Later, 
we will further restrict ourselves to loopless graphs. For standard notions of
Graph Theory we refer the reader to \cite{bondy2008}. In particular, for $n\ge 3$,
 we denote by $P_n$  (resp.\ $C_n$) the path (resp.\ cycle) on $n$ vertices.

Given a pair of graphs $G$ and $H$ a \textit{full-homomorphism}
$\varphi\colon G\to H$ is a vertex mapping such that for each pair of
vertices $x,y\in V(G)$ there is an edge $xy\in E(G)$ if and only if
$\varphi(x)\varphi(y)\in E(G)$. In particular, if $H$ is a simple graph, then
adjacent vertices in $G$ are mapped to different vertices in $H$. Moreover,
if $\varphi(x)  = \varphi(y)$, then $x$ and $y$ have the same neighbourhood
in $G$. 

For a fixed graph $H$, a \textit{full $H$-colouring} of a graph $G$
is a full-homomorphism of $G$ to $H$. A \textit{minimal $H$-obstruction} is a graph
that does not admit a full $H$-colouring, such that every proper induced subgraph
of $G$ admits a full $H$-colouring. We denote by $\obs(H)$ the set of minimal
$H$-obstructions. In~\cite{federDM308}, Feder and Hell showed that for a graph $H$
with $l$ vertices with loops, and $k$ vertices without loops, every graph in $\obs(H)$
has at most $(k+1)(l+1)$ vertices, and this bound is tight. Later, Hell and
Hern\'andez-Cruz showed that the same tight bound holds in the case of
digraphs~\cite{hellDM338}. Independently and in a more 
general setting, Ball, Ne\v{s}et\v{r}il, and Pultr~\cite{ballEJC31}, proved that
for each relational structure $A$, there are a finite number
of minimal $A$-obstructions. Each of these results imply
that for every simple graph $H$ there are finitely many minimal $H$-obstructions.

\begin{proposition}\label{prop:finite-minimal}\cite{ballEJC31,federDM308,hellDM338}
For each graph $H$ there is a finite number of minimal $H$-obstructions.
\end{proposition}

Furthermore, Ball, Ne\v{s}et\v{r}il, and Pultr \cite{ballEJC31} describe the
connected minimal obstructions of paths and cycles, i.e., the connected
graphs in $\obs(C_n)$ and in $\obs(P_n)$. They also propose a
recursive description of disconnected minimal $P_n$-obstructions,
but the ``lists corresponding to the paths do not seem to be more transparent
than those in the connected case'' \cite{ballEJC31}. In this work, we propose
a transparent description of the list of disconnected minimal obstructions
of paths. We do so by means of positive solutions to integer equations. In
particular,  we list all minimal $P_n$-obstructions, and we build on this
description to propose the complete list of minimal obstructions for cycles.

The rest of this work is structured as follows.  First, in
Section~\ref{sec:paths} we propose a description of minimal
$P_n$-obstructions. In Section~\ref{sec:cycles}, we make some
general observations regarding minimal obstructions of regular graphs,
and use these to propose a description of minimal $C_n$-obstructions
in terms of minimal $P_{n-1}$-obstructions. We conclude
this work in Section~\ref{sec:conclusions} where we propose some problems
that arise from observations in Section~\ref{sec:cycles}.
The rest of this section contains some preliminary results needed for this work.

From this point onwards, we only consider loopless finite graphs.
 A pair of vertices $x$ and $y$ of a graph $G$ are called
\textit{false twins}  if $N(x) = N(y)$, and \textit{true twins} of $N[x] = N[y]$. In
particular, every pair of true twins are adjacent, while every pair of false twins are
non-adjacent.  In \cite{sumnerDM5}, Sumner defined a \textit{point-determining} graph
as a graph for which non adjacent vertices have distinct neighbourhoods, i.e., a graph
$G$ is point-determining if it has no pair of false twins.

\begin{proposition}\cite{sumnerDM5}\label{prop:rem-v}
For every non trivial point-determining graph $G$ there is a vertex $v\in V(G)$
such  that $G-v$ is point-determining. Moreover, if $G$ is connected, then 
there are two distinct vertices with that property.
\end{proposition}

A pair of graphs $G$ and $H$ are \textit{full-homomorphically equivalent} if
$G$ admits a full $H$-colouring and $H$ admits a full $G$-colouring.
A \textit{core} in the category of graphs with full-homomorphisms, is a graph
$G$ such that every full-homomorphism $\varphi\colon G\to G$ is surjective.
It is not hard to see that for each graph $H$, there is a unique (up to isomorphism)
core $G$ full-homomorphically equivalent to $H$. In this case, we
say that $G$ is the \textit{full-core} of $H$. 

Point-determining graphs play an important role in the category of graphs with
full-homomorphisms. In particular, every core in the full-homomorphism 
category of graphs is a point-determining graph. Indeed, suppose that $x$ and
$y$ are a pair of false twins in a graph $G$. By mapping $x$ to $y$, we obtain a
full-homomorphism of $G$ onto a proper subgraph, which implies that $G$ is not
a core.  Moreover, the same argument also implies that if $G$ 
is a minimal $H$-obstruction for some graph $H$, then $G$ is a
point-determining graph. Finally,  it is also straightforward to see that if $G$ is
a point-determining graph,  then each full-homomorphism whose domain
is $G$ is an injective mapping. 
The following statement captures two of the facts argued in this paragraph.

\begin{lemma}\label{lem:point-det}
The following statements hold for any pair of graphs $G$ and $H$:
\begin{enumerate}
	\item If $G$ is point-determining, then every full-homomorphism
	$\varphi\colon G\to H$ is injective.
	\item If $G\in \obs(H)$, then $G$ is a point-determining graph.
\hfill $\square$
\end{enumerate}
\end{lemma}

%%%%%%%%%%%%%%%%%%%%%%%%%%%%%%%%%%%%%%%
%%%%%%%%%%%%%%%%%%%%%%%%%%%%%%%%%%%%%%%
%%%%%%%%%%%%%%%%%%%%%%%%%%%%%%%%%%%%%%%
%%%%%%%%%%%%%%%%%%%%%%%%%%%%%%%%%%%%%%%
%%%%%%%%  Path Obstructions %%%%%%%%%%%%%%%%%%%%%
%%%%%%%%%%%%%%%%%%%%%%%%%%%%%%%%%%%%%%%
%%%%%%%%%%%%%%%%%%%%%%%%%%%%%%%%%%%%%%%
%%%%%%%%%%%%%%%%%%%%%%%%%%%%%%%%%%%%%%%
%%%%%%%%%%%%%%%%%%%%%%%%%%%%%%%%%%%%%%%

\section{Path obstructions}
\label{sec:paths}

In this section, we describe the minimal $P$-obstructions when $P$ is a path.
We begin by describing some particular minimal $P$-obstructions. To do so,
we introduce the graphs $A$, $B$ and $E$ depicted in \cref{fig:ABE}.

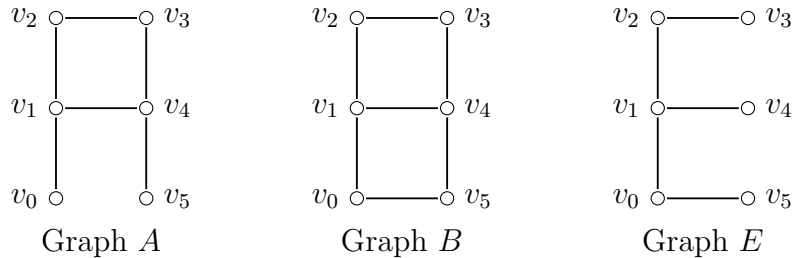
\begin{figure}[ht!]
\begin{center}
\begin{tikzpicture}
[every circle node/.style ={circle,draw,minimum size= 5pt,inner sep=0pt, outer sep=0pt},
every rectangle node/.style ={}];

\begin{scope}[xshift = 4cm, scale=0.4]
\node [circle] (1) at (-1.5,0)[label=left:$v_0$]{};
\node [circle] (2) at (-1.5,3)[label=left:$v_1$]{};
\node [circle] (3) at (-1.5,6)[label=left:$v_2$]{};
\node [circle] (4) at (1.5,6)[label=right:$v_3$]{};
\node [circle] (5) at (1.5,3)[label=right:$v_4$]{};
\node [circle] (6) at (1.5,0)[label=right:$v_5$]{};
\foreach \from/\to in {1/2, 2/3, 1/6, 2/5, 3/4}
\draw [-, shorten <=1pt, shorten >=1pt, >=stealth, line width=.7pt] (\from) to (\to);
\node [rectangle] at (0,-1.5){Graph $E$};

\end{scope}

\begin{scope}[xshift=-4cm, scale = 0.4]
\node [circle] (1) at (-1.5,0)[label=left:$v_0$]{};
\node [circle] (2) at (-1.5,3)[label=left:$v_1$]{};
\node [circle] (3) at (-1.5,6)[label=left:$v_2$]{};
\node [circle] (4) at (1.5,6)[label=right:$v_3$]{};
\node [circle] (5) at (1.5,3)[label=right:$v_4$]{};
\node [circle] (6) at (1.5,0)[label=right:$v_5$]{};
\foreach \from/\to in {1/2,2/3,3/4,4/5,5/6,2/5}
\draw [-, shorten <=1pt, shorten >=1pt, >=stealth, line width=.7pt] (\from) to (\to);
\node [rectangle] (1) at (0,-1.5){Graph $A$};

\end{scope}
\begin{scope}[xshift=0cm, scale=0.4]
\node [circle] (1) at (-1.5,0)[label=left:$v_0$]{};
\node [circle] (2) at (-1.5,3)[label=left:$v_1$]{};
\node [circle] (3) at (-1.5,6)[label=left:$v_2$]{};
\node [circle] (4) at (1.5,6)[label=right:$v_3$]{};
\node [circle] (5) at (1.5,3)[label=right:$v_4$]{};
\node [circle] (6) at (1.5,0)[label=right:$v_5$]{};
\foreach \from/\to in {1/2,2/3,3/4,4/5,5/6,2/5,1/6}
\draw [-, shorten <=1pt, shorten >=1pt, >=stealth, line width=.7pt] (\from) to (\to);
\node [rectangle] (1) at (0,-1.5){Graph $B$};
\end{scope}
\end{tikzpicture}
\end{center}
\caption{For a path $P$, every minimal $P$-obstruction that is neither a linear forest
or a cycle, is one of these graphs (\cref{lem:3possibilties}).}
\label{fig:ABE}
\end{figure}

Recall that for $n \ge 3$, we denote by $C_n$ (resp.\ by $P_n$) the cycle
(resp.\ path) on $n$ vertices; we denote by $K_1$ and $K_2$ the paths on
one and two vertices, respectively. In general,
we denote by $K_n$ the complete graph on $n$ vertices.

\begin{lemma}\label{lem:Cn-On}
For every positive integer $n$, the following statements hold:
\begin{enumerate}
	\item The graph $A$ is a minimal $P_n$-obstruction if and only if $n \ge 6$.
	\item The graph $B$ is a minimal $P_n$-obstruction if and only if $n \ge 5$.
	\item The graph $E$ is a minimal $P_n$-obstruction if and only if $n \ge 7$.
	\item The $m$-cycle is a minimal $P_n$-obstruction if and only if
	$m = 3$ or $5\le m \le n+1$.
\end{enumerate}
\end{lemma}
\begin{proof}
All graphs in statements 1--3 are point-determining graphs that do not
admit a full $P$-colouring for any path $P$. By removing $v_4$ from $A$, we
obtain $K_1+P_4$ which is not full $P_5$-colourable, so $A$ is not a minimal
$P_n$ obstruction for any $n\le 5$. On the other hand, any induced subgraph
of $A$ admits a full $P_6$-colouring, and thus, it is a minimal $P_n$-obstruction
for every $n\ge 6$.  Similarly, $B-v_1$ is not full $P_5$-colourable, and
$E-v_1$ is not full $P_6$-colourable. Also,
every proper induced subgraph of $B$ is full $P_6$-colourable, and
every proper induced subgraph of $E$ is full $P_7$-colourable. Hence,
$B$ is a minimal $P_n$-obstruction if and only if $n\ge 6$, and 
$E$ is a minimal $P_n$-obstruction if and only if $n\ge 7$.  
The last statement is clearly true.
\end{proof}

Now, we observe that each path minimal obstructions is either a graph
mentioned in \cref{lem:Cn-On} or a linear forest.

\begin{lemma}\label{lem:3possibilties}
Consider a path $P$ and a graph $G$. If $G\in \obs(P)$ then one of the
following statements holds:
\begin{enumerate}
	\item $G$ is a cycle.
	\item $G$ is a linear forests.
	\item $G$ is one of the graphs $A$, $B$ or $E$.
\end{enumerate}
\end{lemma}
\begin{proof}
Let $n$ be a positive integer such that  $G\in \obs(P_n)$. We
show that if $G$ is neither a cycle nor a linear forest, then is one of the graphs 
$A$, $B$ or $E$. First, suppose that $G$ is not a cycle or a forest. By minimality of $G$, and
by the fourth statement of \cref{lem:Cn-On}, we know that $G$ does not contain
a triangle nor a cycle of length $m$ with $5\le m \le n+1$. It is not hard to see that
the path on $n+1$ vertices is not full $P_n$-colourable, thus $G$ does not contain
an induced path on $n+1$ vertices, and so, it does not contain a cycle of length
$m\ge n+2$. Putting both of these observations together we conclude that
$G$ contains no triangle nor an induced cycle of length $m\ge 5$. Since
$G$ is not a forest, there is an induced $4$ cycle $C$,  $C = v_1,v_2,v_3,v_4$,
in $G$. 
By the choice of $G$ and by second part of \cref{lem:point-det}, it is the case that
$G$ is a point-determining graph. In particular, $N(v_1) \neq N(v_3)$ and
$N(v_2)\neq N(v_4)$ so, without loss of generality
we assume that $v_1$ has a neighbour $v_0\not\in\{v_2,v_4\}$ and $v_4$ has a
neighbour $v_5\not\in\{v_1,v_3\}$. Since $G$ has no triangles,  the unique neighbour
of $v_0$ (resp.\ $v_5$) in $C$ is $v_1$ (resp.\ $v_4$). 
Let $H$ be the subgraph of $G$ induced by $\{v_0,\dots, v_5\}$. This graph
is isomorphic to either $A$ or $B$. Clearly, neither of $A$ nor $B$ admit a full
$P_n$-colouring, and thus, by minimality of $G$ we conclude that
$G = H$.

In the paragraph above, we showed that if $G\in \obs(G)$ and $G$ is not a forest,
then either $G$ is a cycle  or $G \in\{A,B\}$. To conclude the proof, suppose
that $G$ is a forest but not a linear forest. In this case, $G$ contains an induced 
claw $C$. With a similar procedure to the paragraph above, we extend $C$
to an induced subgraph $H$ of $G$ such that $H\cong E$. 
Since $E$ does not admit a full $P_n$-colouring, we conclude that $G = H \cong E$, 
and the claim follows.
\end{proof}

In order to complete the characterization of $\obs(P_n)$,  we study minimal 
$P_n$-obstructions that are linear forests. To do so, it will be convenient to 
introduce the following notation. First, notice that each linear forest $L$ admits an
injective full-homomorphism to any large enough path. So, we denote by $\mu(L)$
the minimum integer $n$ such that there is an injective full-homomorphism from $L$
to $P_n$.  Also, since linear forests are disjoint unions of paths, we will denote a
linear forests $L$ as $\sum_{k=1}^mP_{n_k}$, where the $k$-th component of $L$
is the path on $n_k$ vertices. Finally, we denote by $c(G)$ the number of connected
component of a graph $G$.

\begin{lemma}\label{lem:LF-injective}
For a linear forest $L=\sum_{k=1}^mP_{n_k}$ the following equalities hold
\[
\mu(L) =  |V(L)| + c(L) -1 = (m-1) + \sum_{k=1}^mn_k.%\vspace{-0.7cm} 
\]
%\hfill $\square$
\end{lemma}
\begin{proof}
One can soon notice that the rightmost equality holds. Now, notice
that if $\varphi\colon L\to P_n$ is an injective full-homomorphism,
then the image of each components of $L$ must be at distance at least
$2$ in $P_n$, thus $n \ge  |V(L)| + c(L) -1$ and so, $\mu(L) \ge  |V(L)| + c(L) -1$. 
It is not hard to see that there is an injective full-homomorphism 
from $L$ to the path on $|V(L)| + c(L) -1$ vertices and hence, $\mu(L)  = |V(L)| + c(L) -1$.
\end{proof}

Consider a linear forest  $L=\sum_{k=1}^mP_{n_k}$. In order to simplify our writing, 
we define $m_i$ to be the number of  components of length $i$ in $L$.
In other words, $m_i$ is the cardinality of the set $\{k\in\{1,\dots,m\}:n_k=i\}$.
In particular, $m_i=0$ for all $i > |V(L)|$. 

Notice that if a linear forest $L$ contains two isolated vertices, then $L$ is not a point-determining
graph. Similarly, if $L$ contains a component isomorphic to $P_3$, then it is also the case that $L$
is not a point-determining graph.

\begin{lemma}\label{lem:LF-components}
Let $L=\sum_{k=1}^mP_{n_k}$ be a linear forest and $n$ a positive integer. 
If $L\in \obs(P_n)$, then the following statements hold:
\begin{enumerate}
	\item $m_1\le 1$.
       \item $n_k\in\{1,2,4,6\}$ for all $k\in\{1,\dots,m\}$.
      \item If $n_k\in\{4,6\}$ for some $k$, then $m_1=1$.
     %\item if $n_k = 6$ for some $k$, then $|V(L)| +m -1 = n+1$.
\end{enumerate}
\end{lemma}
\begin{proof}
By the second part of \cref{lem:point-det}, $L$ is a point-determining graph, so by the
arguments in the paragraph above, we see that $m_1 \le 1$ and there is no $k\in\{1,\dots, m\}$
such that $n_k = 3$. In particular, the first item holds, and to see that the second one
is also true, we show that every component of $L$ has at most
$6$ vertices but not $5$. Anticipating a contradiction, suppose that 
there is a path $P_{n_k} = v_1,v_2,\dots, v_{n_k}$ with $n_k = 5$ or
$n_k\ge 7$, for some $k\in\{1,\dots, m\}$. In such case, $L-v_3$ is a point-determining 
graph and $c(L-v_3)=c(L)+1$ so, by applying \cref{lem:LF-injective} to $L - v_3$ and to
$L$,  we see that $\mu(L-v_3) = \mu(L)$. By the choice of $L$, there is a full
$P_n$-colouring of $L-v_3$ which, by the first part of \cref{lem:point-det}, must be injective.
Thus, by definition of $\mu$, it follows that $n \ge \mu(L-v_3) = \mu(L)$, contradicting
the fact that $L$ is not full $P_n$-colourable. Therefore, if $L\in \obs(P_n)$, then
$n_k\in\{1,2,4,6\}$ for every $k\in\{1,\dots, m\}$.

%The second statement follows because every graph
%with two isolated vertices is not a point-determining graph, but every minimal
%$P_n$-obstruction is a point-determining graph (second part of \cref{lem:point-det}).
%% AQUIII

To prove the third statement, suppose  that $P_{n_k}=v_1,\dots,v_{n_k}$ with 
$n_k\in\{4,6\}$ for some  $k\in\{1, 2, \dots, m\}$. In this case, $c(L-v_2) = c(L)+1$.
So, if $L-v_2$ is a point-determining graph, by using a similar arguments
as in the first paragraph of this proof, we conclude that
$L$ admits a full $P_n$-colouring, contradicting the fact that
$L\in\obs(P_n)$. Hence, $L-v_2$ is not a point-determining linear forest. Since
every component of $L-v_2$ is either a component
of $L$, or $v_1$, or the path $v_3,\dots, v_{n_k}$, it must be the case that
there is an isolated vertex in $L-v_2$ other than $v_1$. Hence, $L$ has at least
one isolated vertex so $m_1\ge 1$, and by the first statement of this lemma, we conclude
 that $m_1 = 1$.

%The final statement follows with arguments similar to those used above. Clearly, 
%if $P_{n_k} = v_1,\dots, v_6$ for some $k\in\{1,\dots, m\}$, then $L-v_1$ is a
%point-determining graph with $c(L-v_1) = c(L)$. Hence, by the first part of 
%\cref{lem:point-det},
%any full-homomorphism from $L-v_1$ to $P_n$ is an injective mapping, and
%thus $c(L-v_1) -1 + |V(L-x)| \le n$.
%Therefore, $c(L) -1 + |V(L)| -1 \le n$ which is equivalent to the inequality
%$|V(L)| +m -1 \le n+1$. We conclude that $c(L) -1 + |V(L)| -1 = n$ using 
%the previous inequality and the leftmost inequality of \cref{lem:LF-bounds}.
\end{proof}

%
%\begin{lemma}\label{lem:LF-bounds}
%Let $L = \sum_{k=1}^mP_{n_k}$ be a linear forest and $n$ a positive integer.
%If $L\in \obs(P_n)$, then
%\[
%n+1 \le (m-1) + \sum_{k=1}^mn_k = (c(L)-1)  + |V(L)| \le n+2.
%\]
%Moreover, if  $L$ has no isolated vertices, then  $(m-1)+|V(L)|-1 =n+1$ equivalently, 
%$c(L) + |V(L)| = n +3$.
%\end{lemma}
%\begin{proof}
%By \cref{lem:LF-injective}, if $(m-1)+\sum_{k=1}^mn_k\le n$ then there
%is a full $P_n$-colouring of $L$. This shows that the inequality $n+1 \le (m-1)+|V(L)|$
%holds. To prove that the second inequality holds, recall that by
%\cref{prop:rem-v} there is  a vertex $x\in V(L)$ such that $L-x$ is point-determining.
%By minimality of $L$, there is an injective full-homomorphism
%$\varphi\colon L-x\to P_n$. Again,
%by Lemma~\ref{lem:LF-injective},
%\[
%c(L-x)-1+|V(L)|-1=c(L-x)-1+|V(L-x)|\le n.
%\]
%By substituting $c(L) -1 \le c(L-x)$ in the inequality above, we observe that
%$c(L)-2+|V(L)|- 1\le n$ and thus, $c(L) + |V(L)| -1 \le n+2$. The first statement
%is now proved. Suppose that $L$ has no 
%isolated vertices. In this case, $c(L)\le c(L-x)$, which implies that
%$c(L)-1+|V(L)| -1 \le n+1$. The claim follows.
%\end{proof}
%
%
%
% 
%
%
It turns out the necessary conditions stated in \cref{lem:LF-components} are almost
sufficient conditions for a linear forest $L$ to be a minimal $P_n$-obstruction.

\begin{proposition}\label{prop:LF-suf-nec}
Let $L=\sum_{k=1}^mP_{n_k}$ be a linear forest and $n$ a positive integer. 
In this case, $L\in\obs(P_n)$ if and only if  one of the following statements holds:
\begin{enumerate}
     \item $\mu(L) = n +1$ and $n_k=2$ for all $k\in\{1,\dots,m\}$.
     \item  $\mu(L) =n + 1$ and $n_k\in\{1,2,4,6\}$ with $m_1 = 1$.
     \item $\mu(L) = n + 2$ and $n_k\in\{1,2,4\}$ with $m_1 = 1$.
\end{enumerate}
\end{proposition}
\begin{proof}
We prove the statement by case distinction depending on the components of $L$, 
and we begin by considering the case when $L = m_2K_2$. One can easily observe
that for every vertex $v$ of $L$, the equality $\mu(m_2K_2 -v) = \mu(m_2K_2) - 1$ holds. 
Also, $L$ and $L-v$  are  point-determining graphs so, $L$ and $L-v$ admit
a full-homomorphism to $P_n$ if and only if they admit an injective full-homomorphism
to $P_n$. Therefore, it follows from the definition of the parameter $\mu$
that $L\in \obs(P_n)$ if and only if $\mu(L) = n+1$. 

Now, suppose that $\mu(L) = n+1$ but $L$ is not a disjoint union of edges. 
By the second part of~\cref{lem:LF-components}, it follows that $n_k\in\{1,2,4,6\}$ for all
$k\in\{1,\dots,m\}$, and by the choice of $L$ and the third part of the same lemma
$m_1 = 1$. Now, we observe that in this case $L$ is a minimal $P_n$-obstruction.
Indeed, if $v$ is an end vertex of any component of $L$, then $\mu(L-v) = \mu(L)-1 = n$, 
so $L-v$ admits a full-homomorphism to $P_n$. Otherwise, if $v$ is a middle vertex of
a $P_4$ or a $P_6$, then $L-v$ is not point-determining: either $L-v$ has two isolated
vertices, or $L-v$ has a component isomorphic to $P_3$. Thus, by identifying the
isolated vertices, or the end vertices of $P_3$, we obtain a full-homomorphism
of $L-v$ to $P_n$. 

If neither of the previous cases holds, then $L$ is not a disjoint union of edges,
and $\mu(L) \le n$ or $\mu(L) \ge n+2$.
In the former case, there is an injective full-homomorphism from $L$ to  $P_n$
so $L$ is not a minimal $P_n$-obstruction. So, if $L\in \obs(P_n)$, then $\mu(L) \ge n +2$.
 Anticipating a contradiction suppose that $L\in \obs(P_n)$ and  that $\mu(L) \ge n+3$.
Again, it must be the case that $L$ has exactly one isolated vertex $v$. One can soon
notice that $\mu(L-v) = \mu(L) -2 \ge n+1$, and that $L-v$ is a point-determining graph.
Thus, by~\cref{lem:point-det}, we conclude that $L-v$ does not admit a full-homomorphism
to $P_n$, contradicting the choice of $L$. Thus if $L\in \obs(P_n)$ and $\mu(L)\neq n+1$, then
$\mu(L) = n +2$. One can easily notice that if $L$ contains a
component isomorphic to $P_6$, and $v$ is an end vertex of this component, then $L-v$ is a
point-determining graph, and $\mu(L-v) = \mu(L) - 1 = n-1$. So, with similar arguments as
before, we conclude that if $L\in \obs(P_n)$ and $\mu(L) = n+2$, then $n_k\in\{1,2,4\}$ 
for all $k\in\{1,\dots, m\}$, and $m_1 = 1$. 
We proceed to observe that if $\mu(L) = n+2$, and $n_k\in\{1,2,4\}$  for all $k\in\{1,\dots, m\}$
with  $m_1 = 1$, then $L\in \obs(P_n)$. Similar as before, if $v$ is the isolated vertex of $L$, 
then $\mu(L-v) = \mu(L) -2 = n$, so $L$ is full $P_n$-colourable. Otherwise, $L-v$
is not point determining: either $L-v$ has two isolated vertices, or a component isomorphic to $P_3$.
Again, by identifying the isolated vertices, or the end vertices of $P_3$, we obtain a full-homomorphism
from $L-v$ to $P_n$. 

The claim now follows because on the one hand,  in the first (resp.\ second and third) paragraph we observed
that the first (resp.\ second and third) statement is a sufficient condition for $L$ to be a minimal $P_n$-obstruction.
On the other one,  every linear forest $L$ satisfies either the first assumption of the first paragraph,
the first assumption of the second paragraph, or the first assumption of the third paragraph. Since in each
of these paragraphs we showed that if $L$ satisfies such assumption and $L\in \obs(P_n)$, then
$L$ must satisfy one of the three items of this proposition; all together proving that if $L\in \obs(P_n)$,
then $L$ satisfies one of~1--3.

\end{proof}

We are ready to propose a description of all minimal $P_n$-obstructions.
To do so, we introduce three sets $C(n)$, $LF(n)$ and $O(n)$,  which depend
on $n$ --- $C$ stands for cycles, $LF$ for linear forests, and $O$ for other. 
We begin with the simplest, 
\[
C(n):= \big\{C_m:~m = 3 \text{ or } 5\le m\le n+1\big\}.
\]
\noindent
Secondly, we define $O(n)$ as follows
\[
O(n):=
\begin{cases}
\varnothing \text{ if } n \le 4,\\ 
\{B\} \text{ if } n = 5,\\ 
\{A,B\} \text{ if } n = 6,\\ 
\{A,B, E\} \text{ if } n \ge 7.\\ 
\end{cases}
\]
\noindent
Finally, $LF(n)$ is the union $LF_1(n)\cup LF_2(n)\cup LF_3(n)$
where
\[
LF_1(n):=\big\{m_2K_2:~ 3m_2 = n+2\big\}
\]
\[
LF_2(n) :=  \big\{K_1 + m_2K_2 + m_4P_4:~ 3m_3 + 5m_4  = n+1\big\}, \text{ and }
\]
\[
LF_3(n) := \big\{K_1 + m_2K_2 + m_4P_4+ m_6P_6:~3m_2 + 5m_4 + 7m_6 = n\big\}.
\]
%\[
%LF(n):=
%\begin{cases}
%\big\{m_2K_2:~ 3m_2 = n+2\big\} \bigcup \big\{K_1 + m_2K_2 + m_4P_4:~ 3m_3 + 5m_4  = n+1\big\} \\
% %\bigcup \big\{K_1 + m_2K_2 + m_4P_4:~ 3m_3 + 5m_4  = n+1\big\}\\ 
% \bigcup \big\{K_1 + m_2K_2 + m_4P_4+ m_6P_6:~3m_2 + 5m_4 + 7m_6 = n\big\}.\\ 
%\end{cases}
%\]
\noindent
We describe the set $\obs(P_n)$ of minimal $P_n$-obstructions in terms of
the previously defined sets. 

\begin{theorem}\label{thm:paths}
For every positive integer $n$ the set $\obs(P_n)$ of minimal $P_n$-obstructions is the
union $ C(n)\cup LF(n) \cup O(n)$.
\end{theorem}
\begin{proof}
It follows from Lemma~\ref{lem:Cn-On}  that $C(n)\cup O(n)\subseteq \obs(P_n)$, and 
from Lemma~\ref{lem:3possibilties} that every $L\in \obs(P_n)\setminus(C(n)\cup O(n))$
is a linear forest.  The fact that the set of linear forests in $\obs(P_n)$ equals $LF(n)$, 
follows from \cref{prop:LF-suf-nec}, and from the equality $\mu(L) = (m-1) + \sum_{k=1}^mn_k$
from \cref{lem:LF-injective}.
\end{proof}

To conclude this section, allow us to discuss an implication of \cref{thm:paths}.
Since all paths are linear forests, any graph that admits a full-homomorphism
to some path, admits a full-homomorphism to some linear forest. On the other
hand, each linear forest admits a full-homomorphism to a large enough path. 
Thus, a graph $G$ admits a full-homomorphism to a path if and only if it admits
a full-homomorphism $G$ to a linear forest.

A \textit{blow-up} of a graph $G$ is obtained by addition of false twins ---
intuitively, by ``blowing up'' some vertices of $G$ to an independent set. 
Clearly, a graph $G$ admits a full $H$-colouring if and only if $G$ is 
a blow-up of some induced subgraph of $H$. Since the class of linear forest
is closed under induced subgraphs, we use the observation in the paragraph
above to prove the following statement.

\begin{corollary}
A graph $G$ is a blow-up of a linear forest if and only if $G$ is
an $\{A,B,E\}$-free graph such that all induced cycles have length four.
\end{corollary}
\begin{proof}
If $G$ is a blow-up of some linear forest, then $G$ admits a full
$P_n$-colouring for some large enough $n$. Thus, by \cref{thm:paths},
$G$ is an $\{A,B,E\}$-free graph such that all induced cycles have length four.
On the other hand, notice that the number of vertices of the smallest graph in 
$LF(n)$ increases with respect to $n$. Thus, for every a graph $G$ 
there is a positive integer $N$ such that $G$ is an $LF(N)$-free graph.
Hence, if $G$ is
an $\{A,B,E\}$-free graph such that all induced cycles have length four,
then $G$ is an $(O(N)\cup C(N) \cup LF(N))$-free graph. 
Therefore, $G$ admits a full $P_N$-colouring, and so, $G$ is
a blow-up of a linear forest.
\end{proof}

%%%%%%%%%%%%%%%%%%%%%%%%%%%%%%%%%%%%%%%
%%%%%%%%%%%%%%%%%%%%%%%%%%%%%%%%%%%%%%%
%%%%%%%%%%%%%%%%%%%%%%%%%%%%%%%%%%%%%%%
%%%%%%%%%%%%%%%%%%%%%%%%%%%%%%%%%%%%%%%
%%%%%%%%%%%%%%% Cycle Obstructions %%%%%%%%%%%%%%
%%%%%%%%%%%%%%%%%%%%%%%%%%%%%%%%%%%%%%%
%%%%%%%%%%%%%%%%%%%%%%%%%%%%%%%%%%%%%%%
%%%%%%%%%%%%%%%%%%%%%%%%%%%%%%%%%%%%%%%
%%%%%%%%%%%%%%%%%%%%%%%%%%%%%%%%%%%%%%%

\section{Cycle obstructions}
\label{sec:cycles}

The aim of this section is listing all minimal obstructions of cycles.
To do so, we first make some general observations
regarding minimal obstructions of regular graphs. \cref{prop:rem-v}
asserts that for each point-determining graph $G$, there is a vertex
$x\in V(G)$ such that $G-x$ is point-determining.  We begin by noticing that
this can be strengthen in the case of regular graphs.

\begin{proposition}\label{prop:nucleus-regular}
Let $H$ be a point-determining graph. If $H$ is a regular graph, then
for each $x\in V(H)$ the graph $H-x$ is point-determining.
\end{proposition}
\begin{proof}
Proceeding by contrapositive, suppose that there is a vertex $x\in V(H)$ such that
$H-x$ is not point-determining. Let $r,s\in V(H-x)$ be a pair of false twins, i.e.,
$rs\not\in E(H-x)$ and $N_{H-x}(r) = N_{H-x}(s)$. 
Since $H$ is a point-determining graph and $rs\not\in E(H)$, it must be the case
that $xr\in E(H)$ and $xs\notin E(H)$ (or vice versa). Hence, $d_H(s) = d_{H-x}(s) =
d_{H-x}(r) =  d_H(r) -1$. Thus, $H$ is not a regular graph.
\end{proof}

Consider a graph $H$ and a minimal $H$-obstruction $G$. By the second
part of \cref{lem:point-det}, $G$ is a point-determining graph so, by
\cref{prop:rem-v}, there is a vertex $v\in V(G)$ such that
$G-v$ is a point-determining graph, and $G-v$ admits a full $H$-colouring
by minimality of $G$.
Also, by  the first part
of  \cref{lem:point-det}, each full-homomorphism from $G-v$ to $H$ is injective,
and thus $|V(G-v)| \le |V(H)|$. Therefore, every graph $G\in \obs(H)$ has at
most $|V(H)|+1$ vertices. We denote by $\obs^\ast(H)$ the set of minimal
$H$-obstructions on $|V(H)|+1$ vertices. The following statement was
proved in \cite{federDM308}.

\begin{proposition}\cite{federDM308}\label{prop:obs*}
For any graph $H$, there are at most two non-isomorphic graphs in $\obs^*(H)$.
\end{proposition}

By similar arguments as in the paragraph above, we observe
that if $G\in \obs^\ast(H)$, then there is a vertex $v\in V(G)$ such that
$G-v\cong H$. 

\begin{observation}\label{obs:iso}
Consider a pair of graphs $G$ and $H$. If $G\in \obs^\ast(H)$, then
there is a vertex $v\in V(G)$ such that $G-v\cong H$.
\end{observation}

This observation about the structure of graphs in $\obs^\ast(H)$
can be strengthened when $H$ is a regular graph. Recall that a pair of
vertices $u$ and $v$ in a graph $G$ are true twins if $N[u] = N[v]$.

\begin{lemma}\label{lem:true-twins}
Let $H$ be a non-complete regular connected graph.
For every graph $G\in \obs^\ast(H)$ there is a pair of true twins
$u,v\in V(G)$ ($u\neq v$)  such that $G-v\cong H$ and $G-u \cong H$.
\end{lemma}
\begin{proof}
Since $H$ is non-complete, it is not isomorphic to $K_2$, and since it is connected,
it is not a matching. Thus, $H$ is a $k$-regular graph with
$k \ge 2$. Let $G\in \obs^\ast(H)$. By \cref{obs:iso}, there is a vertex $x\in V(G)$ such that $G-x\cong H$.
For this proof, it will be convenient to identify $H$ with the subgraph of $G$
induced by $V(G)-\{x\}$. We fix $x$ and use this identification
throughout the proof.  We proceed to show $x$ is not an isolated vertex.
Since $k \ge 2$, there
are no leaves in $H$. Consider a vertex $v\in V(G)-\{x\}$ and let
$\varphi\colon G- v \rightarrow H$ be a full-homomorphism.
Since $H$ is a regular graph, by \cref{prop:nucleus-regular} we know that
$H-v$ is point-determining  so, by the first part of \cref{lem:point-det}, 
the restriction of $\varphi$ to $H-v$ is an injective mapping. Let $L$ be the image
$\varphi[H-v]$ of $H-v$. In particular,  $|V(L) | = |V(H)| -1$. Since $H$ is connected,
either $\varphi(x)$ has a neighbour in $L$ or $\varphi(x)$ belongs to $L$.
Recall, that  $L\cong H-v$ and $H$ has no leaves so, $L$ has no isolated vertices.
Therefore, if $\varphi(x)$ belongs to $L$, then $\varphi(x)$ has a neighbour in $L$,
and since $\varphi$  is a full-homomorphism, $x$ cannot be an isolated vertex
in $G$.

In the paragraph above, we proved that $x$ is not an isolated vertex. Since
$G-x$  is connected (recall that
$G-x = H$), $G$ is a  connected graph. By \cref{prop:rem-v}, there is a vertex
$y\in V(G)-\{x\}$ such that $G-y$ is point-determining, and so, $G-y\cong H$.
We conclude the proof by showing that $x$ and $y$ are true twins in $G$.
Since $H$ is a $k$-regular graph, for each $v\in V(G-y)$ the equality $d_{G-y}(v)=k$ 
holds. Also,  since $H = G-x$, the equality $d_{H-y}(v)=k-1$ holds if and only if
$v\in N_H(y) = N_{G-x}(y)$. On  the other hand, $k-1 = d_{(G-y)-x}(v)-1$ if and only if $v\in N_{G-y}(x)$. Clearly, $(G-y)-x = H- y$, and so,
$v\in N_{G-y}(x)$ if and only if  $v\in N_{G-x}(y)$ for any $v\in V_{G-\{x,y\}}$. Thus, 
$N_G(x)-y=N_{G-y}(x) = N_{G-x}(y)=N_G(y)-x$ so in particular,
$N_G(x)-y = N_G(y)-x$. Since $G$ is point-determining,
$x$ and $y$ are not false twins in $G$ so, $xy\in E(G)$, and thus 
$x$ and $y$ are true twins. The claim follows. 
\end{proof}

\cref{prop:obs*} asserts that $|\obs^\ast(H)|\le 2$ for every graph $H$. Using
\cref{lem:true-twins}, we show that in the case of regular non-complete graphs
 $\obs^\ast(H) = \varnothing$. Recall that a \textit{universal} vertex in a graph $G$
 is a vertex $x\in V(G)$ adjacent to  every $y\in V(G)\setminus\{x\}$.

\begin{proposition}\label{prop:obs*-regular}
For a connected regular graph $H$, the
following equalities hold
\[
\obs^\ast(H) =
\begin{cases}
\{K_1+K_2, K_3\} \text{ if } H \cong K_2,\\
\{K_{n+1}\} \text{ if } H \cong K_n \text{ and }n\neq 2,\\
\varnothing \text{ otherwise.}\\
\end{cases}
\]
\end{proposition}
\begin{proof}
Since the class of complete multipartite graphs is the class of $K_1+K_2$-free graphs,
the class of full $K_n$-colourable graphs is the class of $\{K_1+K_2,K_{n+1}\}$-free
graphs. 
Now, suppose  that $H$ is a regular non-complete connected graph and
let $G\in \obs^\ast(H)$. By \cref{lem:true-twins}, there is a pair 
of true twins $x$ and $y$ of $G$, such that $G-x\cong H \cong G-y$. Again, we identify
$H$ with the subgraph of $G$ induced by $V(G)-x$. Notice that if $x$ and
$y$ are universal vertices in $G$, then $y$ is a universal vertex in $H$ and so,
$H$ is a complete graph (because $H$ is a regular graph). So, by the choice of $H$, 
there is a vertex $z\in V(G)$ such that $zy\notin E(G)$, and since $x$ and $y$
are true twins, it is the case that $xz\notin E(G)$. By the choice of $G$, there is a
full-homomorphism
$\varphi\colon G-z\rightarrow H$. Let $k$ be the degree of every vertex in $H$
so, $d_G(x) = d_G(y)=k+1$. Since $zx,zy\not\in E(G)$, it is the case that
$d_{G-z}(x) = d_{G-z}(y)=k+1$. But $d_H(\varphi(y)) = k$ so,
there are two vertices $r,s\in N_{G-z}(y)$ such that 
$\varphi(r) = \varphi(s)$. Hence $N_{G-z}(r) = N_{G-z}(s)$ and $rs\not \in E(G-z)$.
Recall that $H = G -x$, so $N_{H-z}(r) = N_{H-z}(s)$ and $rs\not \in E(G-z)$, i.e.,
$r$ and $s$ are false twins in $H-z$. Thus, $H-z$ is not a point-determining graph
which contradicts the fact that $H$ is a regular graph and
\cref{prop:nucleus-regular}.
\end{proof}

The following statement shows that if a graph $G$ is a minimal $H$-obstruction
of size $|V(H)|+1$, then every minimal $G$-obstruction $F$ is either a minimal
$H$-obstruction or $|V(F)| = |V(G)| +1$. Conversely, every minimal $H$-obstruction
other than $G$ is a minimal $G$-obstruction.

\begin{theorem}\label{thm:obsH-obsG}
Consider a pair of graphs $H$ and $G$. If $G\in \obs^\ast(H)$, then
\[
\obs(G) = (\obs(H) \setminus\{G\})\cup \obs^\ast(G).
\]
\end{theorem}
\begin{proof}
To simplify notation, let $S = (\obs(H) \setminus\{G\})\cup \obs^\ast(G)$. 
We need to prove that a graph $F$ belongs to $S$ if and only if 
it belongs to $\obs(G)$. Clearly, every graph in $S$ has at most $|V(G)| + 1$
vertices and so does every graph in $\obs(G)$. Moreover, by definition
of $\obs^\ast(G)$, a graph on $|V(G)| +1$ vertices belongs to $\obs(G)$
if and only if it belongs to $\obs^\ast(G)$. Thus, it suffices to prove the claim
for graphs on at most $|V(G)|$ vertices, and by the second part of
\cref{lem:point-det} it suffices to consider point-determining graphs.
We begin by showing that the claim holds for graphs on at most
$|V(G)| -1$ vertices (recall that $|V(G)| = |V(H)| +1$).
Since $G$ is a minimal $H$-obstruction, every proper induced subgraph of $G$
admits a full-homomorphism to $H$. Thus, any graph  that admits a
non-surjective full homomorphism to $G$, admits a full $H$-colouring. 
Hence, a graph on at most $|V(G)| -1$ vertices admits a  full $G$-colouring
if and only if it admits a full $H$-colouring. Therefore a graph on at most
$|V(G)-1|$ vertices belongs to $S$ if and only if it belongs to $\obs(G)$. 

Finally, consider a point-determining graph $L$ on $|V(G)|$ vertices.
By the first part of \cref{lem:point-det}, every full-homomorphism from
$L$ to $G$ is injective, thus $L$ admits a full $G$-colouring if and only
if $L\cong G$.  By similar arguments as in the paragraph above,
we conclude that $L$ is a minimal $G$-obstruction 
if and only if it is a minimal $H$-obstruction.
\end{proof}

\begin{corollary}\label{cor:obs*-determine}
Consider a pair of graphs $H_1$ and $H_2$. If $\obs^\ast(H_1) \cap \obs^\ast(H_2)
\neq\varnothing$, then $H_1\cong H_2$. 
\end{corollary}
\begin{proof}
One soon notices that if $\obs^\ast(H) \neq \varnothing$, then $H$ is a point-determining graph. 
Let $G\in \obs^\ast(H_1) \cap \obs^\ast(H_2)$. By applying \cref{thm:obsH-obsG}
to $H_1$ and $G$, and to $H_2$ and $G$, we conclude that $\obs(H_1) = \obs(H_2)$.
Thus, $H_1$ admits a  full-homomorphism to $H_2$, and vice versa. Since $H_1$ and $H_2$
are point-determining graphs, both full-homomorphisms are injective (\cref{lem:point-det}),
and thus, they are isomorphisms. 
\end{proof}

In other words, \cref{cor:obs*-determine} asserts that if a graph $G$ is
a minimal obstruction of two smaller graphs, then these graphs are isomorphic.
Another immediate implication of \cref{thm:obsH-obsG} is the following one.

\begin{corollary}
Consider a  pair of  graphs $H$ and $G$. If $G\in \obs^\ast(H)$, then there
is at most one minimal $G$-obstruction on $|V(G)|$ vertices.
\end{corollary}

Recall that the orbit of a vertex $y\in V(G)$ is the set of vertices $x\in V(G)$ such
that there  is an automorphism $\varphi\colon G\to G$ such that $\varphi(y) = x$.
We denote the orbit of $y$ by $o(y)$. Clearly, if  $x\in o(y)$, then $G-x\cong G-y$. 

\begin{proposition}\label{prop:v-transitive}
Let $H$ be a non-complete connected vertex transitive graph. For any vertex $x$ 
of $H$, the following equalities hold
\[
obs(H) = obs(H-x) \setminus\{H\} \text{ and } obs(H-x) = obs(H) \cup \{H\}
\]
\end{proposition}
\begin{proof}
Since $H$ is vertex transitive, $H-x\cong H-y$ for any pair of vertices $x,y\in V(H)$. So,
every proper induced subgraph of $H$ admits a full $(H-x)$-colouring.  Since
$H$ is a point-determining graph, $H$ is not full $(H-x)$-colourable, so,
$H\in \obs^\ast(H-x)$. By \cref{thm:obsH-obsG}, we conclude that
$\obs(H) = \obs(H-x) \setminus\{H\} \cup \obs^\ast(H)$ so, using \cref{prop:obs*-regular}
we observe that $\obs^\ast(H) = \varnothing$.
\end{proof}

By applying \cref{prop:v-transitive} to cycles, we see that minimal
$C_n$-obstructions are determined by minimal $P_{n-1}$-obstructions,
and vice versa.

\begin{corollary}\label{cor:cycles}
For every positive integer $n$, $n\ge 5$, the following equalities hold
\[
\obs(C_n) = \obs(P_{n-1})\setminus\{C_n\} \text{ and } 
\obs(P_{n-1}) = \obs(C_n) \cup\{C_n\}.
\]
\end{corollary}

The following characterization of minimal $C_n$-obstructions follows
from \cref{cor:cycles,thm:paths}.

\begin{theorem}\label{thm:cycles}
For every positive integer $n$ the set $\obs(C_n)$ of minimal $C_n$-obstructions
is the union $C(n-2)\cup LF(n-1) \cup O(n-1)$.
\end{theorem}
\begin{proof}
By \cref{cor:cycles}, the equality $\obs(C_n) = \obs(P_{n-1})\setminus \{C_n\}$ holds. 
By \cref{thm:paths}, the set of  minimal $P_n$-obstructions is 
$C(n-1)\cup LF(n-1) \cup O(n-1)$. Finally, by definition of $C(n)$, the
equality $C(n-2) = C(n-1)\setminus\{C_n\}$ holds, and so, the
claim follows.
\end{proof}

\begin{corollary}
A graph $G$ is full $C_5$-colourable if and only if it is
$\{C_3,K_1+P_4,2K_2\}$-free.
\end{corollary}

To conclude this section, we list all $C_n$-minimal obstructions
for small integers  $n$ in Table~\ref{tab:small-n}.

\begin{table}[ht!]
\begin{center}
    \begin{tabular}{| c | l | l |}
    \hline
    $n$ & Linear forests in $\obs(C_n)$ & Other minimal $C_n$-obstructions\\ \hline
    $5$ & $K_1+P_4$ and $2K_2$ & $C_3$\\ \hline
    $6$ &  $K_1+P_4$ and $K_1 + 2K_2$ & $C_3,~C_5$ and $B$\\ \hline
    $7$ &  $K_1+2K_2$ & $C_3,~C_5,~C_6,~A$ and $B$\\ \hline
    $8$ &   $3K_2$, $K_1+K_2+P_4$, and $K_1+P_6$ & $C_3,~C_5,~C_6,~C_7,~A,~B$
     and $E$\\ \hline
    $9$ &   $K_1+3K_2$ and $K_1+ K_2 +P_4$ & $C_3,~C_5,~C_6,~C_8,~A,~B$
     and $E$\\ \hline
    $10$ & $K_1 + 2P_4$ and $K_1 + 3K_2$ &
    $C_3,~C_5,~C_6,~C_7,~C_8,~C_9,~A,~B$ and $E$ \\ \hline
    \end{tabular}
    \caption{To the left, the number of vertices in a cycle $C$. In the middle, 
    the linear forests which are minimal $C$-obstructions. To the right, all
    minimal $C$-obstructions that are not linear forests.}
    \label{tab:small-n}
    \end{center}
  \end{table}

\section{Conclusions}
\label{sec:conclusions}

\cref{prop:obs*-regular} asserts that for a  connected regular graph $H$ the set
$\obs^\ast(H)$ is empty if and only if $H$ is not a complete graph. Also, if $H$ is
obtained from a vertex-transitive graph $G$ by removing one vertex, then
$G\in \obs^\ast(H)$ so, $\obs^\ast(H)\neq \varnothing$. A possible
interesting question to investigate is the following one. 

\begin{question}
Is there a meaningful  characterization of those graphs $H$ for which
$\obs^\ast(H) \neq \varnothing$? 
\end{question}

\cref{thm:obsH-obsG} suggests that there is a close relation between 
a graph $H$ such that $\obs^\ast(H)\neq \varnothing$ and a graph
$G\in \obs^\ast(H)$. 
For this reason, we believe that another possible interesting problem
is determining which graphs $G$ are a minimal $H$-obstruction of
size $|V(H)|+1$ for some graph $H$. 

\begin{question}\label{qst:2}
For which graphs $G$ there is a graph $H$ such that
$G$ is a minimal $H$-obstruction in $\obs^\ast(H)$?
\end{question} 

We briefly observe that this problem is not interesting if we remove the
restriction that $|V(G)| = |V(H)|+1$.

\begin{proposition}
For every point-determining connected graph $G$, there is a graph $H$
such that $G$ is a minimal $H$-obstruction. 
\end{proposition}
\begin{proof}
Let $G$ be as in the hypothesis, and for each vertex $x\in V(G)$ let
$H_x$ be the full-core of $G-x$. Finally, let $H$ be the disjoint union
$\sum_{x\in V(G)}H_x$. Since every connected component of $H$
has less than $|V(G)|$ vertices, there is no injective full-homomorphism
from $G$ to $H$. By the first part of \cref{lem:point-det}, and since
$G$ is a point-determining graph, we conclude that $G$ does not admit
a full $H$-colouring. It is not hard to see that, by the choice of 
$H_x$, for every vertex $x\in V(G)$ there is a full $H$-colouring
of $G-x$. The claim follows.
\end{proof}

The following observation shows that, in the case of regular graphs,
\cref{qst:2} has a meaningful answer. 

\begin{proposition}
Let $G$ be a point-determining regular graph. There is a graph $H$ such that 
$G\in\obs^\ast(H)$ if and only if $G$ is a vertex transitive graph.
\end{proposition}
\begin{proof}
By \cref{prop:nucleus-regular}, if $G$ is a point-determining regular graph,
then for each $x\in V(G)$ the induced subgraph $G-x$ is point-determining.
So, if $|V(G)| = |V(H)|+1$, then by the first part of \cref{lem:point-det}, for
each $x\in V(G)$, every full-homomorphism from $G-x$ to $H$ is an
isomorphism. Hence,  all vertex-deleted subgraphs of $G$ are isomorphic,
and thus $G$ is a vertex transitive graph.
\end{proof}

As a final implication of this work, notice that for every positive
integer $n$, there are at most three graphs in $O(n)$, at most $n-2$ graphs
in $C(n)$, and as many graphs in $LF(n)$ as non-negative solutions
to the diophantine equations, $3x = n+2$, $3x + 5y = n+1$, and
$3x +5y + 7z = n$. It is not hard to observe that there are $O(n^{k-1})$
solutions to each of these equations, where $k$ is the number of variables
in the corresponding equation. 
Hence, there are quadratically many linear forests in $LF(n)$. 
These arguments, together with \cref{thm:paths,thm:cycles}, imply
that the following statement holds.

\begin{corollary}
For every positive integer $n$, there are quadratically
many (with respect to $n$) minimal $P_n$-obstructions and
minimal $C_n$-obstructions.
\end{corollary}

The well-defined and simple structure of paths and cycles might be the reason
why their number of minimal obstructions is polynomially bounded (with respect
to $n$). Nonetheless, having made this observation, it is natural to ask about
the cardinality  of $\obs(H)$ in terms of the  cardinality of the vertex set of $H$.

\begin{question}
Is there a polynomial $p(n)$ such that the size of $\obs(H)$ is bounded
by $p(|V(H)|)$ for each graph $H$?
\end{question}

\section*{Acknowledgements}
The author is deeply grateful to C\'esar Hern\'andez-Cruz for several 
discussions that helped developing this work. In particular, 
C\'esar proposed the problems of describing the minimal path and
cycle obstructions. The author also thanks the anonymous reviewer whose
valuable feedback made the proofs more efficient and easier to read.

\end{document}